\documentclass[10pt]{amsart}

\usepackage{amssymb}
\usepackage{amsthm}
\usepackage{graphics}
\usepackage{amsmath}
\usepackage{amstext}

\usepackage{fancyhdr}

\usepackage{xfrac}
\usepackage{braket}
\usepackage{mathtools}
\usepackage{stmaryrd}
\usepackage{mathrsfs}
\usepackage{extarrows}

\usepackage{microtype}

\usepackage{graphicx}
\usepackage{caption}
\usepackage{subcaption}

\usepackage{todonotes}
\usepackage{import}

\usepackage{ esint }
\usepackage{abstract}

\usepackage{tikz}
\usetikzlibrary{matrix,arrows,calc,patterns}

\title{Bifurcations in the elementary Desboves family}

 \thanks{The first author was partially supported by the ANR project LAMBDA,
 ANR-13-BS01-0002 and by the FIRB2012 grant ``Differential Geometry and Geometric Function Theory", RBFR12W1AQ 002.}

\newtheorem{teo}{Theorem}[section]
\newtheorem{defi}[teo]{Definition}
\newtheorem{cor}[teo]{Corollary}
\newtheorem{lemma}[teo]{Lemma}

\newtheorem{remark}[teo]{Remark}
\newtheorem*{mainthm*}{Main Theorem}

\usepackage{tikz}
\usetikzlibrary{matrix,arrows,calc,shapes,trees,positioning}
\def\b#1{\bar{#1}}
\def\t#1{\tilde{#1}}

\DeclareMathOperator{\jac}{Jac}

\usepackage{enumitem}
\newlist{MA}{enumerate}{1}
\setlist[MA]{label=I.\arabic*}

\newlist{Mintro}{enumerate}{1}
\setlist[Mintro]{label=A.\arabic*}

\newlist{MB}{enumerate}{1}
\setlist[MB]{label=II.\arabic*}

\renewcommand{\P}{\mathbb P}
\newcommand{\C}{\mathbb C}
\newcommand{\lam}{\lambda}
\newcommand{\Dd}{\mathcal D}
\newcommand{\Cc}{\mathcal C}
\newcommand{\pa}[1]{\left(#1\right)}
\newcommand{\Pp}{\mathcal P}
\newcommand{\abs}[1]{\left|#1\right|}
\newcommand{\Ee}{\mathcal E}
\renewcommand{\bra}[1]{\left[#1\right]}
\renewcommand{\bar}{\overline}

\begin{author}[F.~Bianchi]{Fabrizio Bianchi}
\address{ 
 Universit\'e de Toulouse - IMT\\
 UMR CNRS 5219\\
 31062 Toulouse Cedex \\
  France }
  \email{fabrizio.bianchi$@$math.univ-toulouse.fr}
\end{author}

\begin{author}[J.~Taflin]{Johan Taflin}
\address{ 
 Universit\'e de Bourgogne Franche-Comt\'e\\
 UMR CNRS 5584\\
 21078 Dijon Cedex \\
  France }
  \email{johan.taflin$@$u-bourgogne.fr}
\end{author}

\begin{document}

\maketitle

\begin{abstract}
We give an example of a family of endomorphisms of $\P^2(\C)$ whose Julia set depends continuously on the parameter
and whose bifurcation locus has non empty interior.
\end{abstract}

%\tableofcontents

\section{Introduction and result}

It is a classical result due to Lyubich \cite{lyubich1983some} and Man{\'e}-Sad-Sullivan \cite{mane1983dynamics}
that, for a family of rational maps,
the continuity of the Julia set (in the Hausdorff topology) with the parameter
is equivalent to the holomorphic motion of the repelling periodic points and of the Julia set itself.
This allows one to decompose the parameter space of the family into the \emph{stability locus} (where such
holomorphic motions exist)
and its complement, the \emph{bifurcation locus}. De Marco \cite{demarco2001dynamics}
considered the canonical
current $dd^c L$ on the parameter space (where $L$ is the Lyapounov function, given by the integration of the
logarithm of the Jacobian with respect to the equilibrium measure, i.e. the unique measure of maximal entropy)
and proved that it is exactly supported on the bifurcation locus. Thus, stability is equivalent to the harmonicity of
the Lyapounov function. An important feature of this picture is that the stability locus is dense in the parameter space.

This description was recently
extended
to higher dimension by Berteloot, Dupont, and the first author in \cite{bbd2015}. It is proved there
that the complex laplacian of the Lyapounov function
still detects the bifurcation of the repelling cycles and that, up to a zero measure subset with respect to the equilibrium measure,
the Julia
set moves holomorphically precisely out of the support of $dd^c L$. This lead to a coherent definition of stability
and bifurcation in general dimension.
Two main questions were left open in that work: the equivalence between the above stability conditions
and the Hausdorff continuity of the Julia set, and the possible density of stability.
The goal of this paper is to
present an example answering by the negative to both of the above questions. More precisely, we prove the following.

\begin{mainthm*}\label{teo_main}
 The family of endomorphisms of $\P^2$ given by
\[
 f_\lam= \bra{-x(x^3+2z^3) : y(z^3-x^3 +\lam (x^3 + y^3 + z^3)):z(2x^3+ z^3)}
 \]
with $\lam\in \C^*$
satisfies the following properties:
\begin{enumerate}
 \item the Julia set of $f_\lam$ depends continuously on $\lam$, for the Hausdorff topology;
 \item the bifurcation locus coincides with $\C^*$.
\end{enumerate}
\end{mainthm*}

The family above (referred to as the \emph{elementary Desboves family})
was studied by Bonifant-Dabija \cite{bd} and Bonifant-Dabija-Milnor \cite{bdm}, who revealed the dynamical richness
of these maps.
One of their main properties is that the small Julia set $J_2(\lam)$ of $f_\lam$ (the support of the equilibrium measure)
and its large Julia set $J_1(\lam)$ (the complement of the Fatou set, which coincides with
the support of the Green current) coincide (see Theorem \ref{teo_properties}). We can thus refer
to any of them as the \emph{Julia set}.
It would be interesting to know whether the bifurcations in our example are \emph{robust}.
Let us mention that Romain Dujardin has announced the existence of several mechanisms leading to
such phenomena.

The paper is organized as follows. In Section \ref{section_prelim_desboves} we describe in details the Desboves maps
and collect (in Theorem \ref{teo_properties}) the properties that we shall need in the construction, deducing in
particular the continuity of the Julia set and the
presence of bifurcations. This already proves that continuous and holomorphic dependence of the Julia set on the
parameter are not equivalent
in general dimension.
Then, in Section \ref{section_misiurewicz_dense} we prove that the bifurcation locus coincides with $\C^*$, by proving that
\emph{Misiurewicz parameters} (see Definition \ref{defi_misiurewicz}) are dense in the parameter space.
The appendix is devoted to
the proof of Theorem \ref{teo_properties}.
Since
in our particular setting
the proof can be simplified with respect
to the work by Bonifant, Dabija and Milnor,
we include it for sake of completeness.

 \section{Elementary Desboves maps}\label{section_prelim_desboves}
 
 In this section we define and give the main properties of the \emph{elementary Desboves maps}. See \cite{bd} and \cite{bdm}
 for a more detailed description.
 We start considering the classical Desboves map given by
 \[
 f_\Dd := \bra{x(y^3-z^3):y(z^3 -x^3): z(x^3-y^3)}.
 \]
 This is a rational map on $\P^2$, with non-empty indeterminacy locus. The important feature we are interested in is the following:
 the map $f_\Dd$ leaves invariant a pencil of elliptic curves and in particular the \emph{Fermat curve}
 $\Cc :=\{x^3+y^3+z^3=0\}$. We will now consider a well-chosen
 perturbation of $f_\Dd$, still leaving the curve $\Cc$ invariant. A natural way to do this is to consider the following
 \emph{Desboves family} (studied in \cite{bdm}),
 parametrized by $(a,b,c)\in \C^3$. We set $\Phi(x,y,z):= x^3+y^3+z^3$
 for convenience. The Desboves family is thus given by
 \[
 f_{a,b,c} := \bra{x(y^3-z^3 + a\Phi):y(z^3 -x^3 + b\Phi): z(x^3-y^3+ c\Phi)}.
 \]
 By construction, every such map leaves invariant the Fermat curve $\Cc$. Moreover, 
 $f_{a,b,c}$ is actually a well defined endomorphisms of $\P^2$ (i.e., its indeterminacy locus is empty) with $(a,b,c)$
 avoiding
 the union of the seven hyperplanes $abc(a+b+c)(a+1-b) (b+1-c) (c+1-a)=0$.
 
 We will now specialize the parameters in order to recover a last property that we shall need in the sequel: we want every element of
 our perturbation to preserve a pencil of lines (i.e., to be \emph{elementary}, see \cite{dabija_thesis}).
 This can be done by specializing the coefficients $a=-1$ and $c=1$. We thus get the following
 definition (we rename the parameter $b$ to $\lam$).

 \begin{defi}\label{defi_desboves_map}
 Let $\lam\in \C$. An \emph{elementary Desboves map} is a map of the form
 \[
 f_\lam= \bra{-x(x^3+2z^3) : y(z^3-x^3 +\lam (x^3 + y^3 + z^3)):z(2x^3+ z^3)}.
 \]
 The \emph{elementary Desboves family} is the holomorphic family $f:\C^*\times \P^2 \to \C^*\times \P^2$ given by
 $f(\lam,z):= (\lam,f_\lam(z))$.
 \end{defi}

 When $\lam\neq 0$, the map $f_\lam$
 is an endomorphism of $\P^2$.
 With $\lam=0$, the map $f_0$ has
 an indeterminacy locus
 equal to
 the point $\rho_0 :=[0:1:0]$. Moreover,
 we have $f_\lam^{-1} \pa{\{\rho_0\}}=\pa{\{\rho_0\}}$ for every $\lam\in\C^*$. So, $f$
 can also be seen as a family of
endomorphisms
 of $\P^2 \setminus \{\rho_0\}$, parametrized by $\lam\in\C$.
 
We denote by $\Pp$ the pencil of lines passing through the point $\rho_0$. For every $\lam\in \C$, $f_\lam$
preserves $\Pp$ and leaves invariant
the coordinate lines $X:=\{x=0\}, Y:=\{y=0\}, Z:=\{z=0\}$, and
 the Fermat curve $\Cc$.
The line $Y$ is transversal to $\Pp$. Hence, the action on $\Pp$ can be identified to the restriction $g$
of $f_\lam$ to $Y$, which is independent from $\lam.$ To be more precise, by construction
the restriction $h$ of $f_\lam$ to $\Cc$ is independent from $\lam$. It lifts to a map of the
form $z\mapsto \alpha z+\beta$ on $\C$ and a simple computation of the multiplier at a fixed point gives
that $\alpha=-2.$ Hence, $h$ has degree $4$ and is uniformly expanding. If $\pi\colon \P^2\setminus\{\rho_0\}\to Y$
denotes the projection along $\Pp$ on $Y$ then $\pi$ semiconjugates $g$ to $h,$ i.e. $\pi\circ h=g\circ\pi.$
Therefore, the map $g$ is a degree $4$ Latt\`es map. An explicit computation gives that, in the coordinate
$w:= \frac{x}{z}$, the map $g$ takes the form
\begin{equation}\label{eq_g}
w\mapsto -w \frac{w^3+2}{2w^3 +1}.
\end{equation}
It has $3$ critical points, all with double multiplicity, which are given by $[1:0:\omega^j]$,
where $\omega$ is a primitive third-root
of 1 and $j\in\{0,1,2\}$. It is worth pointing out that these three points are mapped by $g$ to the
three (fixed) points of intersections between $Y$ and $\Cc$, $[1:0:-\omega^j]$. Since the intersection of
$\Cc$ and $Y$ are transverse, and the dynamics is expansive on both curves, each of these points is repelling for $f_\lam,$ $\lam\in\C.$

In the next Theorem we collect the properties of the elementary Desboves maps
that we shall need in the sequel.

 \begin{teo}\label{teo_properties}\cite[Proposition 6.16]{bd}
 Let $\lam\in\C^*.$ If $\b R(\lam)$ denotes the closure of the repelling periodic points of $f_\lam$ then
  \begin{enumerate}
   \item $J_1(\lam) = J_2(\lam) = \b R(\lam)$;
   \item the Fatou set coincides with the basin of the superattracting fixed point $\rho_0=[0:1:0]$.
  \end{enumerate}
 \end{teo}

 The first statement in the Main Theorem
 is then a consequence of Theorem \ref{teo_properties}.
 
 \begin{cor}
  Let $\pa{f_\lam}_{\lam\in\C^*}$ be the elementary Desboves family. Then the Julia set depends
  continuously on $\lam$ for the Hausdorff topology.
 \end{cor}

\begin{proof}
As the equilibrium measure depends continuously on $\lam,$ its support $J_2(\lam)$ varies lower semicontinuously
with the parameter for every family of endomorphisms of $\P^2$ (see \cite{ds_cime}).
 It is thus enough to prove that, in the elementary Desboves family, this dependence is also upper semicontinuous.
This is equivalent to prove that the set $\cup_\lam \{\lam\}\times J_2(\lam)$
  is closed in the product space $\C^* \times \P^2$. This immediately follows since its complement $\Omega$
  is the basin of attraction of $\C^* \times \{\rho_0\}$ in the product
  space, which is open. More precisely, if $(\lambda_0,z_0)\in\Omega$ then by Theorem
  \ref{teo_properties}, $z_0$ is in basin of $\rho_0$ for $f_{\lam_0}.$ Therefore, there exists an
  open subset $U$ of $\P^2$ included in this basin such that $\b{f_{\lam_0}(U)}\subset U$ and $z_0\in U$.
  Hence, if $V\subset\C^*$ is a small neighborhood
  of $\lam_0$ then $\b{f_\lam(U)}\subset U$ for all $\lam\in V,$ i.e. $V\times U\subset\Omega.$
  \end{proof}
 
 \begin{remark}
  In our setting, the point $\rho_0$
  has to be thought of as the point at infinity, and the Julia set shares a lot of similarities with the \emph{filled Julia set}
  of a polynomial map. In particular, being the Julia set $J_2(\lam)$ the complement of the ``basin at infinity''
  of $f_\lam$,
  the above proof is essentially the same of the classical
  fact that the filled Julia set depends upper semicontinuously on the parameter.
 \end{remark}

 We study now the critical set $C_\lam$ of the maps $f_\lam$. This is a curve in $\P^2$ of degree 9. More precisely,
 in
 homogeneous coordinates
 the Jacobian of $f_\lam$
 is equal to
 \[
\jac (f_\lam)= 8 (x^3 - z^3)^2 (-x^3 + z^3 + \lam (4y^3 + x^3 + z^3)).
 \]
 The term $(x^3 - z^3)$ gives the three lines $L_j$ passing thought $\rho_0$ and $[1:0:\omega^j]$, with $j\in\{0,1,2\}$, respectively,
 each with multiplicity $2$.
 This corresponds to 
 the critical
 set of
 the Latt\`es map $g.$
The other terms gives a 
(possibly non irreducible) curve of degree 3, given in homogeneous coordinates
by $C'_\lam := \{-x^3 + z^3 + \lam (4y^3 + x^3 + z^3)=0\}$.
We thus have $C_\lam =C'_\lam \cup \pa{L_0 \cup L_1 \cup L_2}$.
Since $C'_\lam$ has degree 3, it intersects $Y$ in three points.
These are given by the points $[x:0:z]$ with $(x,z)$ solution of
$(1+\lam)z^3+ (\lam-1)x^3=0$.
For $\lam=1$, we get the point $[1:0:0]= Z\cap Y$ (as a triple solution). Analogously, for $\lam=-1$
we get the point $[0:0:1]= X\cap Y$. Otherwise, we have the three points $[w:0:1]$, where $w^3 = \frac{1+\lam}{1-\lam}$.
So,
every point $w\in Y$
except the six points $[1:0:\pm \omega^j]$ (i.e., the critical points for the Latt\`es
map and their images, which are always repelling for $f_\lam$,
see above)
there exists a $\lam\neq 0$ such that $w\in C'_\lam \cap Y$.
Moreover, the description above implies the following observation.
\begin{lemma}\label{lemma_w_open}
 Let $V$ be an open subset of $\C$.
 Then the set $W:= \cup_{\lam\in V} C'_{\lam} \cap Y$
 is open in $Y$.
\end{lemma}

For $\lam=0$ the critical set is very special: it is reduced to the union of the three lines $L_i$, each with multiplicity 3. Thus,
the intersection of $C'_0$ with $Y$
is again given by the three points $[1:0:\omega^j]$.\\

Finally,
we study
the two intersection points $x_0:= [0:0:1] = X\cap Y$ and $z_0:= [1:0:0] =Z\cap Y$.
We focus on $x_0$, since the arguments are the same for $z_0$.
Since both the lines $X$ and $Y$ are fixed, the point $x_0$ is a fixed point. Since the restriction of $f$
to $Y$ is a Latt\`es map, the differential of $f$ at $x_0$ has one repelling eigenvalue.
To compute the other one, notice that the restriction of $f_\lam$ to $X$
has the form (in the coordinate $t:= \frac{y}{z}$)
\begin{equation}\label{eq_restr_x}
t\mapsto (1+\lam) t + \lam t^4.
\end{equation}
In particular, the point $\rho_0$ corresponds to the superattracting fixed point at infinity.
The differential at 0 (corresponding to $x_0$) being $(1+\lam)$, we get
that $x_0$ is a repelling fixed point for $\abs{1+\lam}>1$. The analogous computation on $z_0$
gives that $z_0$ is repelling if and only if $\abs{1-\lam}>1$. In particular,
\begin{equation}\label{eq_one_rep}
\mbox{for every $\lam\neq 0$
one among $x_0$ and $z_0$ is repelling.}
 \end{equation}

\begin{remark}
Observe that the transverse multiplier $1+\lam$ at $x_0$ can almost be chosen arbitrarily. If $1+\lam$ satisfies the Brjuno condition, the
 system $f_\lam$
 admits a \emph{Siegel disc} (see \cite{bbd2015})
 passing through $x_0$. By Theorem \ref{teo_properties}, this disc is contained in
 the support of the equilibrium measure. This answers to a question stated in \cite{bbd2015}. A different example of
 this phenomenon was given by the first author in \cite{tesi}. For other values of $\lam,$ it
 can give a parabolic petal or an attracting disc in $J_2(\lam).$ It is even possible to have a set of positive
 Lebesgue measure covered by such attracting discs, see \cite[Theorem 6.3]{bdm} and \cite{taf-elliptic}.
 In particular, the small Julia set has positive measure for these parameters.
\end{remark}

 \section{Misiurewicz parameters and bifurcations}\label{section_misiurewicz_dense}
 
 In this section we prove that the bifurcation locus of the elementary Desboves family coincides with the parameter
 space $\C^*$, completing
 the proof of the Main Theorem.
 By \cite{bbd2015}, \emph{Misiurewicz parameters} (as defined below) are dense in the bifurcation locus. We will thus
 prove that these parameter
 are dense in $\C^*$.
 Recall that $C_f$ and $C_\lam$ denote the critical set of the family $f$ and of $f_\lam$, respectively.
 We shall set $C^n_f:= f^n (C_f)$ and $C_\lam^n := f_\lam^n (C_\lam)$. 

\begin{defi}\label{defi_misiurewicz}
A parameter $\lambda_0 \in  \C^*$ is called a \emph{Misiurewicz parameter} if
there exist
a neighbourhood
$N_{\lambda_0} \subset \C^*$ of $\lambda_0$ and
a holomorphic map $\sigma \colon N_{\lam_0}\to \P^2$
 such that:
 \begin{enumerate}
 \item\label{defi_in_julia_rep} for every $\lambda\in N_{\lam_0}$, $\sigma(\lambda)$ is a
 repelling periodic point;
 \item\label{defi_in_julia} $\sigma(\lam_0)$ is in the Julia set
 of $f_{\lam_0}$;
 \item\label{defi_capta_critico} there exists an $n_0$ such that $(\lambda_0, \sigma (\lambda_0)) \in f^{n_0} (C_f)$;
 \item\label{defi_intersez} $\sigma(N_{\lambda_0}) \nsubseteq f^{n_0} (C_f)$.
 \end{enumerate}
\end{defi}

\begin{teo}
 Let $f$
be the elementary Desboves family. Then
 Misiurewicz parameters are dense in $\C^*$.
\end{teo}

\begin{proof}
 Let $\lam_0\in \C^*$.
 If $V\subset \C^*$ is small enough open neighbourhood of $\lam_0$,
 by \eqref{eq_one_rep}
 one of the two points $[0:0:1]$ and $[1:0:0]$ is repelling.
 Call $N$ this point.
 By Lemma \ref{lemma_w_open}, the set $W:= \cup_{\lam\in V} C'_{\lam} \cap Y$ is open in $Y$.
 Since the preimages of $N$ by $g$ are dense in $Y$, there exists a $\lam_1\in V,$ a point $c_1\in Y$ and a number $n_1$
 such that:
 \begin{enumerate}
  \item $c_1$ is a critical point for $f_{\lam_1}$;
  \item $f_{\lam_1}^{n_1} (c_1) = N$.
 \end{enumerate}
Since $N$ is repelling (and thus contained in the Julia set, by Theorem \ref{teo_properties}),
to prove that $\lam_1$ is a Misiurewicz parameter
it is enough to check that $N$
is not persistently contained in $f^{n_1} (C_f)$. This follows from Lemma \ref{lemma_non_persistent}
below, and the proof is completed as $V$ can be chosen arbitrarily small.
\end{proof}

 \begin{lemma}\label{lemma_non_persistent}
  The points $x_0 = [0:0:1]$ and $z_0 =[1:0:0]$ are not persistently contained in any set $C_\lam^n$.
 \end{lemma}

 \begin{proof}
 Since for every $\lam\in\C$ we have $f_\lam^{-1} (\rho_0) = \{\rho_0\}$,
  we can see the family as a family of endomorphisms of $\P^2 \setminus \{\rho_0\}$, with parameter space $\C$.
  We still denote by $C^n_\lam$ and $C^n_f$
  the iterates of the critical sets.
  
  Assume, for sake of contradition, that there exists a $n_0$ such that $x_0 \in C_{\lam}^{n_0}$
  for every $\lam\in \C^*$. Since $C^{n_0}_f$ is closed, 
  this would imply that $x_0\in C_0^{n_0}$.
  But, for every $n\geq 1$, we have $C_0^{n} \cap Y= \{ [1:0:-\omega^j]\colon j=0,1,2\}$. This gives the desired contradiction.
 \end{proof}

 \begin{remark}
 Theorem 4.6 in \cite{bbd2015} states that, in any holomorphic family $f$ of endomorphisms of $\P^k$,
 the set of parameters $\lam$ for which $\P^k$ coincides with the closure of the postcritical set of $f_\lam$
 is dense in any open subset of the bifurcation locus. It follows that, in the elementary Desboves family, this
 property holds for a dense subsets of parameters in $\C^*$.
 \end{remark}

 \begin{remark}
  It follows from \cite{bbd2015} that, for a dense subset of parameters $\lam\in\C^*$,
  the map $f_\lam$ admits a Siegel disc contained in the support of the equilibrium measure.
 \end{remark}

 \appendix
 \section{Proof of Theorem \ref{teo_properties}}\label{section_appendix}
 
 This section consists of a simplified proof, in our setting, of Theorem \ref{teo_properties}.
 The original proof can be found in \cite{dabija_thesis,bd,bdm}.
 
 \begin{lemma}\label{lemma_exceptional}
  For every $\lam\neq 0$, the exceptional $\Ee$ set of $f_\lam$ is reduced to the point $\rho_0$.
 \end{lemma}

 \begin{proof}
  Since the point $\rho_0$ is backward invariant, it is contained in the exceptional set.
  Let us prove
  the other inclusion.
 The intersection of $\Ee$ with the invariant line $X=\{x=0\}$
  must be reduced to $\rho_0$ (the exceptional set for the restriction, by \eqref{eq_restr_x}).
  So, if $\Ee$ were of dimension 1, it had to be a line passing thought $\rho_0$. This is excluded since
   the exceptional set of
the
Latt\`es restriction $g$ to $Y$ is empty.
Once we know that $\Ee$ has dimension 0, the fact that is reduced to $\rho_0$ again follows since the exceptional set of $g
$ is empty.
\end{proof}

 \begin{lemma}\label{lemma_point_rep_curve_j2}
   There exists a repelling point $r_0$ on the curve $\Cc$ contained in the small Julia set $J_2(\lam)$.
 \end{lemma}

 \begin{proof}
  Since the intersection of $\Cc$ and $Y$
  are transverse, and the dynamics is expansive of both curves, every intersection point
 of $\Cc$ and $Y$ 
 is repelling. Take as $r_0$ any such point (e.g., the fixed point $[1:0:-1]$).
 In order to prove that $r_0\in J_2(\lam)$ we prove that, given any sufficiently small open neighbourhood $U$ of $r_0$, we have
 \begin{equation}\label{eq_claim_key}
 \P^2 \setminus \cup_{n\geq0} f_\lam^n (U) = \{\rho_0\}.
 \end{equation}
 Indeed, take any small ball $V$ disjoint from $\rho_0$ and intersecting the small Julia set $J_2(\lam)$. By \eqref{eq_claim_key},
 preimages of $V$ must accumulate $r_0$. Since $J_2(\lam) \cap V\neq \emptyset$, this proves that $r_0\in J_2(\lam)$, as desired.
 
 Let us thus prove the claim \eqref{eq_claim_key}. Suppose
 that $E:= \P^2 \setminus\pa{\{\rho_0\} \cup \cup_n f_\lam^n (U)} \neq \emptyset$ and
 take a line $l_0$
 through $\rho_0$ intersecting $E$ in a point $e_0$. Since the exceptional set of
 the Latt\`es restriction of $f_\lam$ to $Y$ is empty, there exists a sequence of preimages of the line $l_0$
 converging to the line $l_r$ through $\rho_0$ and $r_0$. This gives a sequence of preimages $e_n$
 of $e_0$ that, up to a subsequence, converge to a point $\t e \in l_r$. Since $f_\lam^{-1} (E)\subset E$ and $E$
 is closed in $\P^2\setminus \{\rho_0\}$, we have
 $\t e \in l_r \cap E$. Moreover, $\t e\neq \rho_0$ since $\rho_0$ is attracting.
  On the other hand, the point $\t e$ must be contained in the exceptional set of the
  restriction of $f_\lam$ to $l_r$. An explicit computation
  yields that this exceptional set is reduced to $\rho_0$ and gives the desired contradiction.
 \end{proof}

 We can now prove Theorem \ref{teo_properties}.
 \begin{proof}[Proof of Theorem \ref{teo_properties}]
  The first assertion is a consequence of Lemma \ref{lemma_point_rep_curve_j2}. 
 Indeed, since backward preimages (in $\Cc$) of any point in $\Cc$ are dense in this curve, this implies
 that all the curve $\Cc$ is included in the Julia set $J_2(\lam)$. Since the exceptional set $\rho_0$ (see Lemma \ref{lemma_exceptional})
 is not contained in $\Cc$,
 preimages of this curve equidistribute (see \cite{ds_cime}) the Green current (supported
 on the large Julia set $J_1(\lam)$).
 This gives that the preimages of $r_0 \in J_2(\lam)$ are dense in $J_1(\lam)$, yielding $J_1(\lam) \subseteq J_2(\lam)$.
 The assertion follows since the inclusions $J_2(\lam) \subseteq \b R(\lam) \subseteq J_1(\lam)$ always hold (the first by Briend-Duval
 equidistribution result \cite{BriendDuval1}).
  
The equality between the Fatou set and the basin
$\Omega$ of $\rho_0$ simply comes from the fact that $f_\lam$ is uniformly
expanding in the transverse direction to $\Pp$ in $\P^2\setminus\Omega.$
\end{proof}

\bibliography{bib_bif_desboves.bib}{}
\bibliographystyle{alpha}

\end{document}